\documentclass[11pt]{article}
\usepackage[centertags]{amsmath}
\usepackage{amsfonts}
\usepackage{amssymb}
\usepackage{amsthm,color}
\usepackage{newlfont}
\usepackage{bbm}

\newtheorem{Theorem}{Theorem}[section]
\newtheorem{Definition}[Theorem]{Definition}

\newtheorem{Lemma}[Theorem]{Lemma}

\newtheorem{Remark}[Theorem]{Remark}

\newtheorem{Hypothesis}{Hypothesis}

\numberwithin{equation}{section}

\begin{document}

\def\le{\left}
\def\r{\right}
\def\cost{\mbox{const}}
\def\a{\alpha}
\def\d{\delta}
\def\ph{\varphi}
\def\e{\epsilon}
\def\la{\lambda}
\def\si{\sigma}
\def\La{\Lambda}
\def\B{{\cal B}}
\def\A{{\mathcal A}}
\def\L{{\mathcal L}}
\def\O{{\mathcal O}}
\def\bO{\overline{{\mathcal O}}}
\def\F{{\mathcal F}}
\def\K{{\mathcal K}}
\def\H{{\mathcal H}}
\def\D{{\mathcal D}}
\def\C{{\mathcal C}}
\def\M{{\mathcal M}}
\def\N{{\mathcal N}}
\def\G{{\mathcal G}}
\def\T{{\mathcal T}}
\def\R{{\mathbb R}}
\def\I{{{\mathbb I}}}

\def\bw{\overline{W}}
\def\phin{\|\varphi\|_{0}}
\def\s0t{\sup_{t \in [0,T]}}
\def\lt{\lim_{t\rightarrow 0}}
\def\iot{\int_{0}^{t}}
\def\ioi{\int_0^{+\infty}}
\def\ds{\displaystyle}
\def\pag{\vfill\eject}
\def\fine{\par\vfill\supereject\end}
\def\acapo{\hfill\break}

\def\beq{\begin{equation}}
\def\eeq{\end{equation}}
\def\barr{\begin{array}}
\def\earr{\end{array}}
\def\vs{\vspace{.1mm}   \\}
\def\rd{\reals\,^{d}}
\def\rn{\reals\,^{n}}
\def\rr{\reals\,^{r}}
\def\bD{\overline{{\mathcal D}}}
\newcommand{\dimo}{\hfill \break {\bf Proof - }}
\newcommand{\nat}{\mathbb N}
\newcommand{\E}{\mathbb E}
\newcommand{\Pro}{\mathbb P}
\newcommand{\com}{{\scriptstyle \circ}}
\newcommand{\reals}{\mathbb R}

\newcommand{\red}[1]{\textcolor{red}{#1}}

\def\Amu{{A_\mu}}
\def\Qmu{{Q_\mu}}
\def\Smu{{S_\mu}}
\def\H{{\mathcal{H}}}
\def\Im{{\textnormal{Im }}}
\def\Tr{{\textnormal{Tr}}}
\def\E{{\mathbb{E}}}
\def\P{{\mathbb{P}}}
\def\span{{\textnormal{span}}}
\title{An averaging approach to the Smoluchowski-Kramers approximation in the presence of a varying magnetic field}
\author{Sandra Cerrai\thanks{Partially supported by the NSF grants DMS 1407615 and DMS 1712934.}\\
\normalsize University of Maryland, College Park, USA
\and
Jan Wehr\thanks{Partially supported by NSF grants DMS 1615045 and DMS 1911358}\\
\normalsize University of Arizona, Tucson, USA
\and
Yichun Zhu\\
\normalsize University of Maryland, College Park, USA}
\date{}

\date{}

\maketitle

\begin{abstract}
We study the small mass limit of the equation describing  planar motion  of a charged particle of a small mass $\mu$ in a force field, containing a magnetic component, perturbed by a stochastic term. We regularize the problem by adding a small friction of intensity $\e>0$. We show that for all small but fixed frictions the small mass limit of $q_{\mu, \e}$ gives the solution $q_\e$ to a stochastic first order equation, containing  a noise-induced drift term. Then, by using a generalization of the classical averaging theorem for Hamiltonian systems by Freidlin and Wentzell, we  take the limit of the slow component of the motion $q_\e$ and we prove that it converges weakly to a Markov process on the graph obtained by identifying all points in the same connected components of the level sets of the  magnetic field intensity function. 

\end{abstract}

\section{Introduction}
We are dealing with  planar motion  of a charged particle of a small mass $\mu$ in a force field, containing a magnetic component, perturbed by a stochastic term
\begin{equation}
\label{sjy1}
\le\{
\begin{array}{l}
\ds{\mu \,\ddot{q}_{\mu}(t)=b(q_\mu(t))-\la(q_{\mu}(t)) A\dot{q}_\mu(t)+\si(q_\mu(t))\,\dot{w}_t,}\\
\vs
\ds{q_\mu(0)=q \in\,\mathbb{R}^2,\ \ \ \ \ \dot{q}_\mu(0)=p \in\,\mathbb{R}^2.}
\end{array}\r.
\end{equation}
Here $b$ is a vector field in $\mathbb{R}^2$, $\si$ is $2\times 2$-matrix-valued mapping defined on $\mathbb{R}^2$ and $w(t)$ is a standard two-dimensional Brownian motion. 
Moreover, $\la:\mathbb{R}^2\to \mathbb{R}$ is a mapping, such that $ \la(x)\geq \la_0>0$, for every $x \in\,\mathbb{R}^2$, and
\[A=\le(\begin{array}{cc}
0 & 1\\
-1 & 0\end{array}\right).\]
We are  interested in understanding the limiting  behavior of the solution $q_\mu$ to the equation \eqref{sjy1}, as the mass $\mu$ converges to $0$. This is the so called Smoluchowski-Kramers limit.

\bigskip

It is well known (see \cite{f} for all details) that when the variable magnetic field considered in the present paper is replaced by a constant friction (that is $\la$ is constant and the matrix $A$ is the identity matrix),  then $q_\mu(t)$ converges to the solution of the first order equation 
\begin{equation}
\label{sjy2}
dq(t)=b(q(t))\,dt+\si(q(t))\,dw(t),\ \ \ \ q(0)=q.
\end{equation}
More precisely, for every fixed $T>0$ 
\begin{equation}
\label{sjy5}
\lim_{\mu\to 0} \mathbb{E}\max_{t \in\,[0,T]}|q_\mu(t)-q(t)|^2=0.\end{equation}
Here $q_\mu$ can be a vector of any finite dimension. The same result can be obtained also if $\la$ is still constant, but $A$ is a more general matrix, whose eigenvalues have strictly positive real parts, with the limiting equation \eqref{sjy2} replaced by
\begin{equation}
\label{sjy7}
dq(t)=A^{-1}b(q(t))\,dt+A^{-1}\si(q(t))\,dw(t),\ \ \ \ q(0)=q.\end{equation}

The  case of non constant friction has been widely studied recently (see \cite{bw} and \cite{hmdvw} for example).
To summarize the relevant results, consider the system
\begin{equation}
\label{sjy3}
\le\{
\begin{array}{l}
\ds{\mu \,\ddot{q}_{\mu}(t)=b(q_\mu(t))-\gamma(q_{\mu}(t))\dot{q}_\mu(t)+\si(q_\mu(t))\,\dot{w}_t,}\\
\vs
\ds{q_\mu(0)=q \in\,\mathbb{R}^k,\ \ \ \ \ \dot{q}_\mu(0)=p \in\,\mathbb{R}^k,}
\end{array}\r.
\end{equation}
for an $h$-dimensional Brownian motion $w(t)$. Assume that  the coefficients $b:\mathbb{R}^k\to \mathbb{R}^k$, $\gamma:\mathbb{R}^k\to   \mathbb{R}^{k\times k}$ and $\si:\mathbb{R}^k\to \mathbb{R}^{h\times k}$ are smooth and uniformly bounded and the smallest eigenvalue $\la_1(q)$ of the symmetric matrix $\gamma(q)+\gamma^\star(q)$ is  strictly positive, uniformly in $q \in\,\mathbb{R}^k$,
\[\inf_{q \in\,\mathbb{R}^k} \la_1(q)=:\bar{\la}>0.\]
In this case  the relation \eqref{sjy5}  still holds, but now $q(t)$ is the solution of the modified equation
\begin{equation}
\label{sjy6}
dq(t)=\le[\gamma^{-1}(q(t))b(q(t))+S(q(t))\r]\,dt+\gamma^{-1}(q(t))\si(q(t))dw(t),\ \ \ \ \ q(0)=q.
\end{equation}
Here $S(q)$ is the {\em noise-induced drift} whose $j$-th component equals
\[S_j(q)=\sum_{i,l=1}^k \frac{\partial}{\partial q_i}(\gamma^{-1})_{jl}(q) J_{li}(q),\ \ \ \ j=1,\ldots,k,\]
where $J$ is the matrix-valued function solving the Lyapunov equation
\[J(q)\gamma^\star(q)+\gamma(q) J(q)=\sigma(q) \sigma^\star(q),\ \ \ \ \ q \in\,\mathbb{R}^k.\]

In \cite{cf11}, the case of a particle subject to a constant  magnetic field orthogonal to the plane where the particle moves has been considered.  In this case, the motion of the particle is governed by equation \eqref{sjy1}, with $\la(q)\equiv \bar{\la}$, for every $q \in\,\mathbb{R}^2$ (for simplicity of notation in what follows we shall take $\bar{\la}=1$). In particular, since the eigenvalues of $A$ are purely imaginary, the methods and results described above are not applicable in this case.

It is not difficult to check  that if the stochastic term in \eqref{sjy1} is replaced by a continuous function, then $q_\mu$ converges uniformly in $[0,T]$ to the solution of \eqref{sjy7}. But for the white noise the result is different: while
\[\lim_{\mu\to 0}\int_0^t \sin \frac s\mu \,\varphi(s)\,ds=0,\]
for every continuous function $\varphi$, when $w(t)$ is a Brownian motion we have
\[\text{Var}\le(\int_0^t \sin \frac s\mu \,dw(s)\r)=\int_0^t\sin^2 \frac s\mu\,ds\to \frac t2,\ \ \ \ \text{as}\ \mu \downarrow 0,\]
so that
\[\lim_{\mu\to 0}\int_0^t \sin \frac s\mu \,dw(s)
\neq 0.\]

Because of this, in \cite{cf11} the problem has been regularized, in order to prove a suitable analog of the Smoluchowski-Kramers approximation.
The first regularization consisted in introducing in equation \eqref{sjy1} a small friction proportional to the velocity. Namely, the following equation has been considered
\[
\le\{\begin{array}{l}
\ds{\mu\, \ddot{q}_{\mu,\e}(t)=b(q_{\mu,\e}(t))-A_\e \dot{q}_{\mu,\e}(t)+\si(q_{\mu,\e}(t))\dot{w}(t),}\\
\vs
\ds{q_{\mu,\e}(0)=q \in\,\reals^{2},\ \ \ \ \dot{q}_{\mu,\e}(0)=p \in\,\reals^{2},}
\end{array}\r.
\]
where  
$
A_\e=A+\e\,I$
and $\e>0$ is a small parameter.
It has been  shown that for any $T>0$ 
\begin{equation}
\label{intro24}
\lim_{\mu\to 0}\,\E\max_{t \in\,[0,T]}|q_{\mu,\e}(t)-q_\e(t)|^2=0,
\end{equation}
where $q_\e(t)$ is the solution of the problem
\[
dq(t)=A_\e^{-1}b(q(t))\,dt+A_\e^{-1}\si(q(t))dw(t),\ \ \ q(0)=q.
\]
Next,  it has been shown that
\[\lim_{\e\to 0}\,\E \max_{t \in\,[0,T]}|q_\e(t)-q(t)|^2=0,
\]
where $q(t)$ is the solution of the problem
\begin{equation}
\label{intro3}
dq(t)=-A\, b(q(t))\,dt-A\,\si(q(t))\,dw(t),\ \ \ \ \ \ q(0)=q.\end{equation}

Another approach to regularization (see also \cite{lee} for the case of nonconstant magnetic field) used the fact that the white noise $\dot{w}(t)$ can be considered as an idealization of an isotropic $\d$-correlated smooth mean-zero Gaussian process $\dot{w}^\d(t)$, with
$0<\d<< 1$, which converges to $\dot{w}(t)$, as $\d\downarrow 0$. In this case, it has been proven that if $q_{\mu,\d}(t)$ is the solution of equation \eqref{sjy1},  with $\dot{w}(t)$ replaced by $\dot{w}^\d(t)$, then 
\[
\lim_{\mu\to 0}\,\E\max_{t \in\,[0,T]}|q_{\mu,\d}(t)-q_{\d}(t)|=0,\]
where $q_{\d}(t)$ solves the equation
\[
\dot{q}(t)=-A b(q(t))-A\si(q(t))\,\dot{w}^\d(t),\ \ \ q(0)=q.
\]
Next, by taking the limit as $\d\downarrow 0$, it has been proven that $q_\d(t)$ converges to the solution $\hat{q}(t)$ of the problem 
\[
d\hat{q}(t)=-A b(\hat{q}(t))\,dt-A\si(\hat{q}(t))\,\circ  dw(t),\ \ \ \hat{q}(0)=q,
\]
where  the stochastic term is now interpreted in Stratonovich sense.

\bigskip

In the present paper we are interested in the small mass limit in  the presence of a nonconstant magnetic field. To this purpose we add a small constant friction and we consider the regularized equation
\begin{equation}
\label{intro1000}
\le\{
\begin{array}{l}
\ds{\mu \,\ddot{q}_{\mu, \e}(t)=b(q_{\mu, \e}(t))-\le[\la(q_{\mu, \e}(t))A+\e I\r] \dot{q}_{\mu, \e}(t)+\si(q_{\mu, \e}(t))\,\dot{w}_t,}\\
\vs
\ds{q_{\mu, \e}(0)=q \in\,\mathbb{R}^2,\ \ \ \ \ \dot{q}_{\mu, \e}(0)=p \in\,\mathbb{R}^2.}
\end{array}\r.
\end{equation}
We show that under suitable conditions on the coefficients $b$, $\si$ and $\la$, the  above problem is well posed in $L^k(\Omega;C([0,T];\mathbb{R}^2))$, for every $T>0$ and $k\geq 1$.

For every fixed $\e>0$, equation \eqref{intro1000} is of the same type as those considered in \cite{hhv} and \cite{hmdvw}, so that we can take the small mass limit as $\mu$ goes to zero, obtaining
\[
\lim_{\mu\to 0}\mathbb{E}\sup_{t \in\,[0,T]}|q_{\mu, \e}(t)-q_\e(t)|=0,
\]
where $q_\e$ is the solution of the problem
\[\le\{\begin{array}{l}
\ds{dq_\e(t)=\le[\left(\la(q_\e(t)A+\e I\r)^{-1} b(q_\e(t))+S_\e(q_\e(t))\r]\,dt+\left(\la(q_\e(t)A+\e I\r)^{-1}\si(q_\e(t))\,dw(t),}\\
\vs
\ds{q_\e(0)=q.}
\end{array}\r.\]
After some computations, we obtain
\[\begin{array}{l}
\ds{dq_\e(q)=
\frac 1\e\gamma(q_\e(t)) {\nabla}^\perp\la(q_\e(t))\,dt+B(q_\e(t))\,dt+\Sigma(q_\e(t))\,dw(t),}\\
\vs
\ds{+\e\,\le[B_\e(q_\e(t))\,dt+\Sigma_\e(q_\e(t))\,dw(t)\r],\ \ \ \ \ \ q_\e(0)=q,}
\end{array}\]
where the mappings $\gamma:\mathbb{R}^2\to\mathbb{R}$, $B, B_\e:\mathbb{R}^2\to\mathbb{R}^2$ and $\Sigma, \Sigma_\e:\mathbb{R}^2\to \mathbb{R}^{2\times 2}$ can be explicitly computed.
This means that the motion of $q_\e$ is a superposition  of a fast component on the level sets of $\la$ and a slow transversal motion. Using a suitable generalization of the classical result of Freildin and Wentzell on averaging for Hamiltonian systems (see \cite[Chapter 8]{fw} and \cite{Y}), we shall prove that the projection of $q_\e$ over the graph $\Gamma$, obtained by identifying all points on the same connected component of each level set of $\la$, converges to a  Markov process $Y$, whose generator is explicitly given.

\section{Well-posedness of the regularized problem}

As  mentioned in the Introduction, we are dealing here with the  equation
\begin{equation}
\label{sjy10}
\le\{
\begin{array}{l}
\ds{\mu \,\ddot{q}_{\mu}(t)=b(q_\mu(t))-\la(q_{\mu}(t)) A\dot{q}_\mu(t)+\si(q_\mu(t))\,\dot{w}_t,}\\
\vs
\ds{q_\mu(0)=q \in\,\mathbb{R}^2,\ \ \ \ \ \dot{q}_\mu(0)=p \in\,\mathbb{R}^2,}
\end{array}\r.
\end{equation}
where $\mu$ is a small positive constant and $w(t)$ is a standard Brownian motion in $\mathbb{R}^2$.

In this section, we shall assume that the coefficients of the equation above satisfy the following assumptions. In Section \ref{sec3} we will impose a more restrictive growth condition on $\lambda$.

\begin{Hypothesis}
\label{H1}
\begin{enumerate}
\item The mappings $b:\mathbb{R}^2\to \mathbb{R}^2$ and $\si:\mathbb{R}^2\to \mathbb{R}^{2\times 2}$ are Lipschitz-continuous.
\item The mapping   $\la:\mathbb{R}^2\to\mathbb{R}$ is locally Lipschitz-continuous and there exist $\gamma\geq 0$ and $c>0$ such that
\begin{equation}
\label{sjy18}
|\la(q)|\leq c\le(1+|q|^\gamma\r),\ \ \ \ \ \la \in\,\mathbb{R}^2.
\end{equation}
Moreover
\[\inf_{q \in\,\mathbb{R}^2} \la(q)=:\la_0>0.\]
\end{enumerate}
\end{Hypothesis}

Next, for every $\e\geq 0$ we introduce the regularized problem
\begin{equation}
\label{sjy11}
\le\{
\begin{array}{l}
\ds{\mu \,\ddot{q}_{\mu, \e}(t)=b(q_{\mu, \e}(t))-\La_\e(q_{\mu, \e}(t)) \dot{q}_{\mu, \e}(t)+\si(q_{\mu, \e}(t))\,\dot{w}_t,}\\
\vs
\ds{q_{\mu, \e}(0)=q \in\,\mathbb{R}^2,\ \ \ \ \ \dot{q}_{\mu, \e}(0)=p \in\,\mathbb{R}^2,}
\end{array}\r.
\end{equation}
where
\[\La_\e(q)=\la(q)A+\e I=\le(\begin{array}{cc}
\e&\la(q)\\
-\la(q)  &  \e
\end{array}\r),\ \ \ \ q \in\,\mathbb{R}^2.\]
Notice that for every $\e>0$ the matrix $\La_\e(q)$ is uniformly non-degenerate, as
\begin{equation}
\label{sjy19}
\langle \La_\e(q)p,p\rangle=\e\,|p|^2.
\end{equation}
Moreover, when $\e=0$, equation \eqref{sjy11} coincides with equation \eqref{sjy10}.

\begin{Theorem}
Under Hypothesis \ref{H1}, for every $\mu>0$ and $\e\geq 0$ and for every $T>0$ and $k\geq 1$, equation \eqref{sjy11} admits a unique adapted solution $q_{\mu, \e} \in\,L^k(\Omega;C([0,T];\mathbb{R}^2))$.
\end{Theorem}

\begin{proof}
For every $q, p \in\,\mathbb{R}^2$ and $n \in\,\mathbb{N}$, we define
\[\beta_n(p)=\ds{\begin{cases}
p,  &  \text{if}\  |p|\leq n,\\
\vs
\ds{np/|p|,}  &  \text{if}\ |p|\geq n,
\end{cases}}\]
and
\[\La_{\e, n}(q)=\la_n(q)A+\e I,\ \ \ \ \text{where}\ \ \ \ \la_{n}(q)=\ds{\begin{cases}
\la(q),  &  \text{if}\  |q|\leq n,\\
\vs
\ds{\la((n+1)q/|q|),}  &  \text{if}\ |q|\geq n+1,
\end{cases}}.\]
Notice that $\la_n:\mathbb{R}^2\to \mathbb{R}$ is Lipschitz-continuous and 
\begin{equation}
\label{sjy15}
|\la_n(q)|\leq c\le(1+|q|^\gamma\r),\ \ \ |\beta_n(p)|\leq |p|,
\end{equation}
for some constant $c$ independent of $n$.
Moreover, since $\langle A\beta_n(p),p\rangle=0,$ and $\langle \beta_n(p),p\rangle\leq |p|^2$, 
for every $p \in\,\mathbb{R}^2$ and $n \in\,\mathbb{N}$, we have  
\begin{equation}
\label{sjy13}
\langle \La_{\e, n}(q)\beta_n(p),p\rangle= \e\,|p|^2,\end{equation}
for every $p, q \in\,\mathbb{R}^2$, $n \in\,\mathbb{N}$ and $\e>0$.

With these notations, we introduce the problem
\[\le\{
\begin{array}{l}
\ds{\mu \,\ddot{q}^n_{\mu, \e}(t)=b(q^n_{\mu, \e}(t))-\La_{\e, n}(q^n_{\mu, \e}(t)) \beta_n(\dot{q}^n_{\mu, \e}(t))+\si(q^n_{\mu, \e}(t))\,\dot{w}_t,}\\
\vs
\ds{q^n_{\mu, \e}(0)=q \in\,\mathbb{R}^2,\ \ \ \ \ \dot{q}^n_{\mu, \e}(0)=p \in\,\mathbb{R}^2,}
\end{array}\r.
\]
which can be rewritten as
\begin{equation}
\label{sjy12}
\left\{\begin{array}{l}
\ds{dq_{\mu, \e}^n(t)=p_{\mu, \e}^n(t)\,dt,\ \ \ \ \ q^n_{\mu, \e}(0)=q}\\
\vs
\ds{\mu dp_{\mu, \e}^n(t)=\le[b(q_{\mu, \e}^n(t))-\La_{\e, n}(q_{\mu, \e}^n(t))\beta_n p_{\mu, \e}^n(t)\r]+\si(q_{\mu, \e}^n(t))\,dw(t),\ \ \ \ \ p_{\mu, \e}^n(t0)=p.}
\end{array}\r.\end{equation}
It is immediate to check that, for every fixed $n \in\,\mathbb{N}$ and $\e>0$, the mapping 
\[(q,p) \in\,\mathbb{R}^2\times \mathbb{R}^2\mapsto \La_{\e, n}(q)\beta_n(p) \in\, \mathbb{R}^2,\]
is Lipschitz-continuous, so that  equation \eqref{sjy12} admits a unique adapted solution $(q_{\mu, \e}^n, p_{\mu, \e}^n) \in\,L^p(\Omega;C^1([0,T];\mathbb{R}^2)\times C([0,T];\mathbb{R}^2))$. 

Now, applying It\^o's formula to the function $\Phi(q,p)=|q|^{2k}+|p|^{2k}$, for $k\geq 2$, we obtain
\[\begin{array}{l}
\ds{|q_{\mu, \e}^n(t)|^{2k}+|p_{\mu, \e}^n(t)|^{2k}=|q|^{2k}+|p|^{2k}+k\int_0^t |q_{\mu, \e}^n(s)|^{2k-2}\langle q_{\mu, \e}^n(s),p_{\mu, \e}^n(s)\rangle\,ds}\\
\vs
\ds{+\frac k\mu \int_0^t |p_{\mu, \e}^n(s)|^{2k-2}\langle p_{\mu, \e}^n(s),b(q_{\mu, \e}^n(s)-\La_{\e, n}(q_{\mu, \e}^n(s))\beta_n(p_{\mu, \e}^n(s))\rangle\,ds}\\
\vs
\ds{+\frac k{2\mu^2} \int_0^t |p_{\mu, \e}^n(s)|^{2k-2}\mbox{Tr}\le[\si \si^\star(q_{\mu, \e}^n(s))\r]\,ds+\frac {k(k-1)}{2\mu^2}  \int_0^t |p_{\mu, \e}^n(s)|^{2k-4}|\si(q_{\mu, \e}^n(s))p_{\mu, \e}^n(s)|^2\,ds}\\
\vs
\ds{+ \frac k\mu \int_0^t |p_{\mu, \e}^n(s)|^{2k-2}\langle p_{\mu, \e}^n(s),\si(q_{\mu, \e}^n(s))\,dw(s)\rangle.}
\end{array}\]
Therefore, thanks to \eqref{sjy13} and to the Young inequality, we have  for every $\e>0$
\[\begin{array}{l}
\ds{|q_{\mu, \e}^n(t)|^{2k}+|p_{\mu, \e}^n(t)|^{2k}\leq |q|^{2k}+|p|^{2k}+c_{k,\mu}\int_0^t \le[|q_{\mu, \e}^n(s)|^{2k}+|p_{\mu, \e}^n(s)|^{2k}\r]\,ds}\\
\vs
\ds{+\frac k\mu \int_0^t |p_{\mu, \e}^n(s)|^{2k-2}\langle p_{\mu, \e}^n(s),\si(q_{\mu, \e}^n(s))\,dw(s)\rangle.}
\end{array}\]
Taking  expected  in values of both sides and using the Gronwall lemma, we obtain
\begin{equation}
\label{sjy14}
\mathbb{E}|q_{\mu, \e}^n(t)|^{2k}+\mathbb{E}|p_{\mu, \e}^n(t)|^{2k}\leq c_{k,\mu}(T)\le(1+|q|^{2k}+|p|^{2k}\r),\ \ \ \ t \in\,[0,T].
\end{equation}
Therefore,
since 
\[q_{\mu, \e}^n(t)=q+\int_0^t p_{\mu, \e}^n(s)\,ds,\]
and
\[p_{\mu, \e}^n(t)=p+\frac 1\mu\int_0^t\le[b(q_{\mu, \e}^n(s))-\La_{\e, n}(q_{\mu, \e}^n(s))\beta_n(p_{\mu, \e}^n(s))\r]\,ds+\frac 1\mu\int_0^t \si(q_{\mu, \e}^n(s))\,dw(s),\]
using \eqref{sjy15}, from \eqref{sjy14} we obtain
\begin{equation}
\label{sjy17}
\sup_{n \in\,\mathbb{N}}\mathbb{E}\sup_{t \in\,[0,T]}\le(|q_{\mu, \e}^n(t)|^{2k}+|p_{\mu, \e}^n(t)|^{2k}
\r)\leq c_{k,\mu}(T,|q|, |p|).
\end{equation}

Now, for any $n \in\,\mathbb{N}$ we define
\[\tau_n=\inf \le\{t\geq 0\ :\ |q_{\mu, \e}^n(t)|\vee |p_{\mu, \e}^n(t)|\geq n\r\},\]
with the usual convention that $\inf \emptyset =+\infty$. Since 
\begin{equation}
\label{cwz1}
(q_{\mu, \e}^n(t),p_{\mu, \e}^n(t))=(q_{\mu, \e}^{m}(t),p_{\mu, \e}^{m}(t)),\ \ \ \ n<m,\ \ \ t\leq \tau_n,\end{equation}
it follows that 
the sequence $\{\tau_n\}_{n \in\,\mathbb{N}}$ is non-decreasing, $\mathbb{P}$-a.s., so that we can define 
\[\tau=\lim_{n\to \infty}\tau_n.\]
Due to \eqref{sjy17}, for every fixed $T>0$ we have
\[\begin{array}{l}
\ds{\mathbb{P}\le(\sup_{t \in\,[0,T]}|q_{\mu, \e}^n(t)\leq n,\ \sup_{t \in\,[0,T]}|p_{\mu, \e}^n(t)\leq n\r)}\\
\vs
\ds{\geq 1-\mathbb{P}\le(\sup_{t \in\,[0,T]}|q_{\mu, \e}^n(t)> n\r)-\mathbb{P}\le(\sup_{t \in\,[0,T]}|p_{\mu, \e}^n(t)> n\r)}\\
\vs
\ds{\geq 1-\frac{2 c_{1,\mu}(T,|q|, |p|)}{n}.}
\end{array}\]
This implies  
\[\lim_{n\to\infty}\mathbb{P}\le(\tau_n>T\r)=1,\]
and then, since $T$ is arbitrary, we conclude
\[\mathbb{P}\le (\tau=+\infty\r)=1.\]
In particular, if we set
\[(q_{\mu, \e}(t),p_{\mu, \e}(t))=(q_{\mu, \e}^n(t\wedge \tau_n),p_{\mu, \e}^n(t\wedge \tau_n)),\ \ \ \ t\leq  \tau,\]
due to \eqref{cwz1} we can conclude that there exists a unique solution $(q_{\mu, \e},p_{\mu,\e})$ to problem \eqref{sjy11}, belonging to $L^k(\Omega;C^1([0,T];\mathbb{R}^2)\times C([0,T];\mathbb{R}^2))$, for every $k \geq 1$ and $T>0$.
\end{proof}

\section{The Smoluchowski-Kramers approximation for the regularized problem}

For every $\e>0$ and $q \in\,\mathbb{R}^2$, the matrix $\La_\e(q)$ is invertible and  
\begin{equation}
\label{sjy33}
\La_\e^{-1}(q)=\frac 1{\la^2(q)+\e^2}\le(\begin{array}{cc}
\e&-\la(q)\\
\la(q)  &  \e
\end{array}\r).
\end{equation}
We introduce the vector field $S^\e(q)$, whose $j$-th component is defined by
\begin{equation}
\label{sjy22}
S^\e_j(q)=\sum_{i,l=1}^2 \partial_i(\La_{\e}^{-1})_{jl}(q) J^\e_{li}(q),\ \ \ \ j=1, 2,
\end{equation}
where  $\partial_i=\partial/\partial q_i$ and  $J^\e$ is the matrix-valued function solving the Lyapunov equation
\[J^\e(q)\La_\e^\star(q)+\La_\e(q) J^\e(q)=\sigma(q) \sigma^\star(q),\ \ \ \ \ q \in\,\mathbb{R}^2.\]
Thanks to \eqref{sjy19}, the equation above has a unique solution $J^\e$  which can be explicitly written as 
\begin{equation}
\label{cwz2}
\begin{array}{l}
\ds{J^\e(q)=\int_0^\infty e^{-\La_\e(q)r}\si\si^\star(q) e^{-\La_\e^\star(q)r}\,dr}\\
\vs
\ds{=\int_0^\infty e^{-\la(q)A r}\si\si^\star(q) e^{\la(q) Ar}e^{-2\e r}\,dr,\ \ \ \ \ q \in\,\mathbb{R}^2.}
\end{array}\end{equation}
It is immediate to check that
\[e^{-\la(q)A r}=\left(\begin{array}{cc}
\cos (\la(q)r)  &  -\sin (\la(q)r)\\
\vs
\sin (\la(q)r)  &  \cos (\la(q) r)
\end{array}\r),\ \ \ \ r\geq 0.\]
In what follows,  for every $q \in\,\mathbb{R}^2$ we denote
\[\left(\begin{array}{cc}
a_1(q)  &  a_0(q)\\
\vs
a_0(q)  &  a_2(q)
\end{array}\r)=:\si \si^\star(q),\]
and
\begin{equation}
\label{sjy50}
\beta_0(q):=\frac{a_1(q)+a_2(q)}4,\ \ \ \ \beta_1(q):=\frac{a_1(q)-a_2(q)}4\ \ \ \ \beta_2(q):=\frac{a_0(q)}2.
\end{equation}

\begin{Lemma}
\label{l2.2}
Assume that $\la:\mathbb{R}^2\to \mathbb{R}$ is differentiable. Then, there exist $M:\mathbb{R}^2\to \mathbb{R}^{2\times 2}$ and $R^\e:\mathbb{R}^2\to \mathbb{R}^{2\times 2}$ such that for every $\e>0$
\begin{equation}
\label{sjy30}
S^\e(q)=\frac 1\e\frac{\beta_0(q)}{\la^2(q)}{\nabla}^\perp\la(q)-M(q)\nabla \la(q)+R^\e(q)\nabla \la(q),\ \ \ \ \ q \in\,\mathbb{R}^2.
\end{equation}
\end{Lemma}

\begin{proof}
Thanks to \eqref{cwz2},
we have \[\begin{array}{l}
\ds{J^\e_{11}(q)=\frac{\beta_0(q)}{\e}+\beta_1(q)\int_0^\infty \cos (\la(q) r)e^{-\e r}\,dr-\beta_2(q)\int_0^\infty \sin(\la(q) r)e^{-\e r}\,dr}\\
\vs
\ds{J^\e_{22}(q)=\frac{\beta_0(q)}{\e}-\beta_1(q)\int_0^\infty \cos (\la(q) r)e^{-\e r}\,dr+\beta_2(q)\int_0^\infty \sin(\la(q) r)e^{-\e r}\,dr}\\
\vs
\ds{J^\e_{12}(q)=J^\e_{21}(q)=\beta_1(q)\int_0^\infty \sin(\la(q)r)e^{-\e r}\,dr+\beta_2(q)\int_0^\infty \cos(\la(q)r)e^{-\e r}\,dr.}
\end{array}
\]
Integrating by parts, we obtain
\[\int_0^\infty \cos (\la(q) r)e^{-\e r}\,dr=\frac{\e}{\la^2(q)+\e^2},\]
and
\[\int_0^\infty \sin(\la(q) r)e^{-\e r}\,dr=\frac{\la(q)}{\la^2(q)+\e^2}.\]\
This implies
\begin{equation}
\label{sjy20}
\begin{array}{l}
\ds{J^\e_{11}(q)=\frac{\beta_0(q)}{\e}+\beta_1(q)\frac{\e}{\la^2(q)+\e^2}-\beta_2(q)\frac{\la(q)}{\la^2(q)+\e^2}}\\
\vs
\ds{J^\e_{22}(q)=\frac{\beta_0(q)}{\e}-\beta_1(q)\frac{\e}{\la^2(q)+\e^2}+\beta_2(q)\frac{\la(q)}{\la^2(q)+\e^2}}\\
\vs
\ds{J^\e_{12}(q)=J^\e_{21}(q)=\beta_1(q)\frac{\la(q)}{\la^2(q)+\e^2}+\beta_2(q)\frac{\e}{\la^2(q)+\e^2}.}
\end{array}
\end{equation}

Now, due to \eqref{sjy33},
for every $\e>0$ and $q \in\,\mathbb{R}^2$ we have
\begin{equation}
\label{sjy21}
\begin{array}{l}
\ds{\partial_i\le(\La_\e^{-1}\r)_{11}(q)=\partial_i\le(\La_\e^{-1}\r)_{22}(q)=-\frac{2\e\la(q)}{(\la^2(q)+\e^2)^2}\,\partial_i\la(q),\ \ \ \ i=1,2,}\\
\vs
\ds{
\partial_i \le(\La_\e^{-1}\r)_{12}(q)=-\partial_i\le(\La_\e^{-1}\r)_{21}(q)=\frac{\la^2(q)-\e^2}{(\la^2(q)+\e^2)^2}\,\partial_i\la(q),\ \ \ \ i=1,2.}
\end{array}
\end{equation}
Substituting \eqref{sjy20} and \eqref{sjy21} in \eqref{sjy22}, we obtain
\[\begin{array}{l}
\ds{S^\e_1(q)=-\frac{2\e\la(q)}{(\la^2(q)+\e^2)^2}\le[\le(\frac{\beta_0(q)}{\e}+\beta_1(q)\frac{\e}{\la^2(q)+\e^2}-\beta_2(q)\frac{\la(q)}{\la^2(q)+\e^2}\r)\partial_1\la(q)\r.}\\
\vs
\ds{+\le.\le(\beta_1(q)\frac{\la(q)}{\la^2(q)+\e^2}+\beta_2(q)\frac{\e}{\la^2(q)+\e^2} \r)\partial_2\la(q)\r]}\\
\vs
\ds{+\frac{\la^2(q)-\e^2}{(\la^2(q)+\e^2)^2}\le[\le(\beta_1(q)\frac{\la(q)}{\la^2(q)+\e^2}+\beta_2(q)\frac{\e}{\la^2(q)+\e^2}\r)\partial_1\la(q)\r.}\\
\vs
\ds{\le.+\le(\frac{\beta_0(q)}{\e}-\beta_1(q)\frac{\e}{\la^2(q)+\e^2}+\beta_2(q)\frac{\la(q)}{\la^2(q)+\e^2}\r)\partial_2\la(q)\r],}
\end{array}\]
and
\[\begin{array}{l}
\ds{S^\e_2(q)=-\frac{\la^2(q)-\e^2}{(\la^2(q)+\e^2)^2}\le[\le(\frac{\beta_0(q)}{\e}+\beta_1(q)\frac{\e}{\la^2(q)+\e^2}-\beta_2(q)\frac{\la(q)}{\la^2(q)+\e^2}\r)\partial_1\la(q)\r.}\\
\vs
\ds{+\le.\le(\beta_1(q)\frac{\la(q)}{\la^2(q)+\e^2}+\beta_2(q)\frac{\e}{\la^2(q)+\e^2} \r)\partial_2\la(q)\r]}\\
\vs
\ds{-\frac{2\e\la(q)}{(\la^2(q)+\e^2)^2}\le[\le(\beta_1(q)\frac{\la(q)}{\la^2(q)+\e^2}+\beta_2(q)\frac{\e}{\la^2(q)+\e^2} \r)\partial_1\la(q)\r.}\\
\vs
\ds{\le.+\le(\frac{\beta_0(q)}{\e}-\beta_1(q)\frac{\e}{\la^2(q)+\e^2}+\beta_2(q)\frac{\la(q)}{\la^2(q)+\e^2}\r)\partial_2\la(q)\r],}
\end{array}\]

Now, we define
\[\Gamma^\e_1(q):=\beta_1(q)\frac{\la(q)}{\la^2(q)+\e^2}+\beta_2(q)\frac{\e}{\la^2(q)+\e^2},\ \ \ \ \ \ \ \Gamma_1(q):=\frac{\beta_1(q)}{\la(q)},\]
and
\[\Gamma^\e_2(q):=\beta_1(q)\frac{\e}{\la^2(q)+\e^2}-\beta_2(q)\frac{\la(q)}{\la^2(q)+\e^2},\ \ \ \ \ \ \ \Gamma_2(q):=-\frac{\beta_2(q)}{\la(q)}.\]

With these notations
\[
\begin{array}{l}
\ds{S^\e_1(q)=\frac 1\e \frac{\beta_0(q)}{\la^2(q)}\,\partial_2\la(q)+\le[-\frac{2\beta_0(q)}{\la^3(q)}+\frac{\Gamma_1(q)}{\la^2(q)}\r]\,\partial_1\la(q)-\frac{\Gamma_2(q)}{\la^2(q)}\,\partial_2\la(q)}\\
\vs
\ds{+R_{11}^\e(q)\partial_1\la(q)+R_{12}^\e(q)\partial_2\la(q),}
\end{array}\]
where
\begin{equation}
\label{sjy24}
\begin{array}{l}
\ds{R^\e_{11}(q):=
-\frac{2\e\la(q)}{(\la^2(q)+\e^2)^2} \Gamma^\e_2(q)+\e\beta_2(q)\frac{\la^2(q)-\e^2}{(\la^2(q)+\e^2)^3}}\\
\vs
\ds{+
2\,\beta_0(q)\le[\frac 1{\la^3(q)}-\frac{\la(q)}{(\la^2(q)+\e^2)^2}\r]-\beta_1(q)\le[\frac 1{\la^3(q)}-\frac{\la(q)\le(\la^2(q)-\e^2\r)}{(\la^2(q)+\e^2)^3}\r],}
\end{array}\end{equation}
and
\begin{equation}
\label{sjy25}
\begin{array}{l}
\ds{R^\e_{12}(q):=-\frac{2\la(q)\e}{(\la^2(q)+\e^2)^2}\Gamma_1^\e(q)
-\e\beta_2(q)\frac{\la^2(q)-\e^2}{(\la^2(q)+\e^2)^3}}\\
\vs
\ds{-\frac{\beta_0(q)}\e\le[\frac 1{\la^2(q)}-\frac{\la^2(q)-\e^2}{(\la^2(q)+\e^2)^2}\r]+\beta_1(q)\le[\frac 1{\la^3(q)}-\frac{\la(q)\le(\la^2(q)-\e^2\r)}{(\la^2(q)+\e^2)^3}\r].}
\end{array}
\end{equation}
Similarly, we have 
\[
\begin{array}{l}
\ds{S^\e_2(q)=-\frac 1\e \frac{\beta_0(q)}{\la^2(q)}\partial_1\la(q)-\frac{\Gamma_2(q)}{\la^2(q)}\partial_1\la(q)-\le[\frac{2\beta_0(q)}{\la^3(q)}+\frac{\Gamma_1(q)}{\la^2(q)}\r]\partial_2 \la(q)}\\
\vs
\ds{+R^\e_{21}(q)\partial_1\la(q)+R^\e_{22}(q)\partial_2\la(q),}
\end{array}
\]
where
\begin{equation}
\label{sjy27}
\begin{array}{l}
\ds{R^\e_{21}(q):=-\frac{2\la(q)\e}{(\la^2(q)+\e^2)^2}\Gamma_1^\e(q)-\e\beta_1(q)\frac{\la^2(q)-\e^2}{(\la^2(q)+\e^2)^3}}\\
\vs
\ds{+\frac{\beta_0(q)}\e\le[\frac 1{\la^2(q)}-\frac{\la^2(q)-\e^2}{(\la^2(q)+\e^2)^2}\r]-\beta_2(q)\le[\frac 1{\la^3(q)}-\frac{\la(q)\le(\la^2(q)-\e^2\r)}{(\la^2(q)+\e^2)^3}\r],}
\end{array}
\end{equation}
and 
\begin{equation}
\label{sjy28}
\begin{array}{l}
\ds{R^\e_{22}(q):=\frac{2\e\la(q)}{(\la^2(q)+\e^2)^2}\Gamma^\e_2
-\e\beta_2(q)\frac{\la^2(q)-\e^2}{(\la^2(q)+\e^2)^3}}\\
\vs
\ds{+2\beta_0(q)\le[\frac 1{\la^3(q)}-\frac{\la(q)}{(\la^2(q)+\e^2)^2}\r]+\beta_1(q)\le[\frac 1{\la^3(q)}-\frac{\la(q)\le(\la^2(q)-\e^2\r)}{(\la^2(q)+\e^2)^3}\r].}
\end{array}
\end{equation}
Therefore, recalling that $\Gamma_1(q)=\beta_1(q)/\la(q)$ and $\Gamma_2(q)=-\beta_2(q)/\la(q)$, if we define
\begin{equation}
\label{sjy31}
M(q)=\frac 1{\la^3(q)}\le(\begin{array}{cc}
\ds{2\beta_0(q)-\beta_1(q) } & \ds{-\beta_2(q)}\\
& \vs
\ds{-\beta_2(q)}  &  \ds{2\beta_0(q)+\beta_1(q)}
\end{array}\r),\end{equation}
and $R^\e(q)=(R^\e_{ij}(q))_{i,j=1,2}$, where the components $R^\e_{ij}(q)$ are defined in \eqref{sjy24}, \eqref{sjy25}, \eqref{sjy27} and \eqref{sjy28}, we obtain \eqref{sjy30}.
\end{proof}

In what follows we shall assume that the following condition is satisfied.

\begin{Hypothesis}
\label{H2}
\begin{enumerate}
\item The mapping $\la:\mathbb{R}^2\to \mathbb{R}$ is continuously differentiable.
\item For every $\e>0$, the mapping   $S_\e:\mathbb{R}^2\to \mathbb{R}^2$ introduced in \eqref{sjy22}  is locally Lipschitz-continuous and has linear growth.
\item For every $\e>0$ the mappings $\La_\e^{-1}b:\mathbb{R}^2\to \mathbb{R}^2$ and $\La_\e^{-1}\si :\mathbb{R}^2\to\mathbb{R}^{2\times 2}$ are locally Lipschitz-continuous and have linear growth.
\end{enumerate}
\end{Hypothesis}

\begin{Remark}
\label{rem2.3}
{\em 
\begin{enumerate}
\item AUsing the explicit  expression of $M(q)$ given in \eqref{sjy31} and the expressions for the coefficients of $R_\e(q)$ given in \eqref{sjy24}, \eqref{sjy25}, \eqref{sjy27} and \eqref{sjy28}, thanks to what we have already assumed in Hypothesis \ref{H1}, we can check easily that Hypothesis \ref{H2} is satisfied if we assume $\si$ to  be bounded and $\la$ to be bounded and differentiable, with $\nabla \la:\mathbb{R}^2\to \mathbb{R}^2$ Lipschitz-continuous.
\item  On the other hand, if we assume that $\nabla \la:\mathbb{R}^2\to \mathbb{R}^2$ is locally Lipschitz-continuous and has linear growth and there exists $c>0$ such that for $|q|$ large enough
\[|\la(q)|\geq c\,|q|^2,\]
then Hypothesis \ref{H2} is satisfied, without assuming $\si$ to be bounded.
\end{enumerate}
}
\end{Remark}

\begin{Theorem}
For every $\mu, \e>0$, let $q_{\mu, \e}$  be the solution of problem \eqref{sjy12}.
 Then, under Hypotheses \ref{H1} and \ref{H2}, for every $\e>0$ we have
\begin{equation}
\label{sjy35}
\lim_{\mu\to 0}\mathbb{E}\sup_{t \in\,[0,T]}|q_{\mu, \e}(t)-q_\e(t)|=0,
\end{equation}
where $q_\e$ is the solution of the problem
\begin{equation}
\label{sjy34}
dq_\e(t)=\le[\La_\e^{-1} b(q_\e(t))+S_\e(q_\e(t))\r]\,dt+\La_\e^{-1}\si(q_\e(t))\,dw(t),\ \ \ \ \ q_\e(0)=q.
\end{equation}

\end{Theorem}

\begin{proof}
Assuming Hypotheses \ref{H1} and \ref{H2}, we have that for every $\e>0$ and for every $k\geq 1$ and $T>0$  problem \eqref{sjy34}
admits a unique solution $q_\e \in\,L^k(\Omega;C([0,T];\mathbb{R}^2))$.
As $\langle \La_\e(q)p,p\rangle=\e\,|p|^2$, this allows to conclude the proof thanks to  \cite[Theorem 2.4]{hhv}.
\end{proof}

\section{The averaging limit}
In this section we want to investigate the limiting behavior of the slow component of $q_\e$, as $\e$ goes to zero. To this purpose, we need to introduce some preliminary material.

\label{sec3}

\subsection{Some notations and further assumptions}

We consider here the  system 
\begin{equation}
\label{ham}
\dot{X}(t)=\frac{\beta_0(X(t))}{\la^2(X(t))}\,\nabla^\perp  \la(X(t)).
\end{equation}
Clearly, for every $t\geq 0$, we have $\la(X(t))=\la(X(0))$. Now, if we consider the perturbed system
\[\begin{array}{l}
\ds{dX_\e(q)=
\frac{\beta_0(X_\e(t))}{\la^2(X_\e(t))}{\nabla}^\perp\la(X_\e(t))\,dt}\\
\vs
\ds{+ \e\le[\frac{1}{\la(X_\e(t))}Ab(X_\e(t))-M(X_\e(t))\nabla \la(X_\e(t))\r]\,dt+\frac {\sqrt{\e}}{\la(X_\e(t))}A\si(X_\e(t))\,dw(t)}\\
\vs
\ds{+\e^2\le[H^\e(X_\e(t))b(X_\e(t))+\hat{R}^\e(X_\e(t))\nabla \la(X_\e(t))\r]\,dt+\e H^\e(X_\e(t))\si(X_\e(t))\,dw(t),}
\end{array}\]
the quantity $\la(X_\e(t))$ is no longer conserved. However, for any fixed time interval $[0,T]$ and for every $k\geq 1$,  we have
\[\lim_{\e\to 0}\mathbb{E}\sup_{t \in\,[0,T]}|X_\e(t)-X(t)|^k=0,\]
and, as an immediate consequence,
\[\lim_{\e\to 0}\mathbb{E}\sup_{t \in\,[0,T]}|\la(X_\e(t))-\la(X(0))|^k=0.\]

Now, with the change of time $t\mapsto t/\e$, we can check that
\[\mathcal{L}(X_\e(\cdot/\e))=\mathcal{L}(q_\e(\cdot)),\]
where $q_\e$ is the solution of equation \eqref{sjy34}.
As  mentioned above, our aim is to identify the non trivial limit for the distribution of the process $\la(q_\e(\cdot))$, as $\e\downarrow 0$. To this purpose, in addition to Hypotheses \ref{H1} and \ref{H2}, we assume that $\la$ satisfies the following conditions.

\begin{Hypothesis}
\label{H3}
\begin{enumerate}
\item If $\beta_0$ is the function defined in \eqref{sjy50}, we have
\begin{equation}
\label{sjy51}
\inf_{x \in\,\mathbb{R}^2}\beta_0(x)>0.
\end{equation}
\item The mapping $\la:\mathbb{R}^2\to\mathbb{R}$ is four times continuously differentiable, with bounded second derivative.
\item The mapping $\la$ has only a finite number of critical points $x_1,\ldots, x_n$. The matrix of second derivative $D^2\la(x_i)$ is non degenerate, for every $i=1,\ldots,n$ and $\la(x_i)\neq \la(x_j)$,  if $i\neq j$.
\item There exists a positive constant $a$ such that
$\la(x)\geq a\,|x|^2$, $|\nabla \la(x)|\geq a\,|x|$ and $\Delta \la(x)\geq a$, for all $x \in\,\reals^2$, with $|x|$ large enough.

\end{enumerate}
\end{Hypothesis}

\begin{Remark}
{\em 
 Remember that  the function $\beta_0$ was defined as $[(\si\si^\star)_{11}^2+(\si\si^\star)_{22}^2]/4$. Therefore, condition \eqref{sjy51} is a non-degeneracy condition on the noisy perturbation. 
}

\end{Remark}

Next, for every $z\geq \la_0$, we denote by $C(z)$ the $z$-level set 
\[C(z)=\le\{x \in\,\mathbb{R}^2\,:\,\la(x)=z\r\}.\]
The set  $C(z)$ may consist of several connected components
\[C(z)=\bigcup_{k=1}^{N(z)} C_k(z),\]
and for every $x \in\,\mathbb{R}^2$  we have
\[X(0)=x\Longrightarrow X(t) \in\,C_{k(x)}(\la(x)),\ \ \ \ t\geq 0,\]
where $C_{k(x)}(x)$ is the  connected component of the level set $C(\la(x))$, to which  the point $x$ belongs.
For every $z \geq 0$ and $k=1,\ldots,N(z)$, we shall denote by $G_k(z)$ the domain of $\mathbb{R}^2$ bounded by the level set component $C_k(z)$.

\medskip

If we identify all points in $\mathbb{R}^2$  belonging to the same connected component of a given level set $C(z)$ of the Hamiltonian $\la$, we obtain a graph $\Gamma$,  given by several edges $I_1,\ldots I_n$ and vertices $O_1,\ldots, O_m$. The vertices will be of two different types,  external and internal vertices. External vertices correspond to local extrema of  $\la$, while internal vertices correspond to saddle points of $\la$. Among external vertices, we will also include $O_\infty$, the vertex of the graph corresponding to the point at infinity.

In what follows, we shall denote by $\Pi:\mathbb{R}^2\to \Gamma$ the {\em identification map}, that associates to every point $x \in\,\mathbb{R}^2$ the corresponding point $\Pi(x)$ of the graph $\Gamma$. We have $\Pi(x)=(\la(x),k(x))$, where $k(x)$ denotes the number of the edge on the graph $\Gamma$, containing the point $\Pi(x)$. If $O_i$ is one of the interior vertices, the second coordinate cannot  be chosen in a unique way, as there are three edges having $O_i$ as their endpoint. Notice that both $k(x)$ and $H(x)$ are  first integrals (a discrete and a continuous one, respectively) of  system \eqref{ham}.

On the graph $\Gamma$, a distance can be introduced in the following way. If $y_1=(z_1,k)$ and $y_2=(z_2,k)$ belong to the same edge $I_k$, then $d(y_1,y_2)=|z_1-z_2|$. In the case $y_1$ and $y_2$ belong to different edges, then
\[d(y_1,y_2)=\min\,\le\{d(y_1, O_{i_1})+d(O_{i_1},O_{i_2})+\cdots+d(O_{i_j},y_2)\r\},\]
where the minimum is taken over all possible paths from $y_1$ to $y_2$, through every possible sequence of vertices $O_{i_1},\ldots,O_{i_{j}}$, connecting $y_1$ to $y_2$.

\medskip
If $z$ is not a critical value, then each $C_k(z)$ consists of one periodic trajectory of the vector field $\nabla^\perp \la(x)$. If $z$ is a local extremum of $\la(x)$, then, among the components of $C(z)$ there is a set consisting of one point, the equilibrium point of the flow. If $\la(x)$ has a saddle point at some point $x_0$ and $\la(x_0)=z$, then $C(z)$ consists of three trajectories: the equilibrium point $x_0$ and the two trajectories that have $x_0$ as their limiting point, as $t\to \pm \infty$.

\medskip

Now, for every $(z,k) \in\,\Gamma$, we define
\begin{equation}
\label{fluid2}
T_k(z)=\oint_{C_k(z)}\frac {\la^2(x)}{\beta_0(x)|\nabla \la(x)|}\,dl_{z,k},
\end{equation}
where $dl_{z,k}$ is the length element on $C_k(z)$. Notice that $T_k(z)$ is the period of the motion along the level set $C_k(z)$.

As is seen above, if $X(0)=x \in\,C_k(z)$, then $X(t) \in\,C_k(z)$, for every $t\geq 0$. As  known, for every $(z,k) \in\,\Gamma$ the probability measure
\begin{equation}
\label{brr}
d\mu_{z,k}:=\frac 1{T_k(z)}\,\frac {\la^2(x)}{\beta_0(x)|\nabla \la(x)|}\,dl_{z,k}\end{equation}
is invariant for system \eqref{ham} on the level set $C_k(z)$.

\subsection{The  limit of $\Pi(q_\e)$}
Due to \eqref{sjy33}, for every $\e>0$ we have
\begin{equation}
\label{sjy40}
\La_\e^{-1}(q)=\frac 1{\la(q)}A+\e\,H^\e(q),
\end{equation}
where
\[H^\e(q):=\frac{1}{\la^2(q)+\e^2}\le(I-\frac{\e}{\la(q)}A\r).\]
Notice that
\begin{equation}
\label{sjy38}
\sup_{\e>0}|H^\e(q)|<\infty,\ \ \ \ \ q \in\,\mathbb{R}^2.
\end{equation}

\begin{Lemma}
Let ${R}^\e:\mathbb{R}^2\to \mathbb{R}^{2\times 2}$ be the mapping introduced in Lemma \ref{l2.2}. Then
\begin{equation}
\label{sjy39}
\sup_{\e>0}\frac 1\e\,|\hat{R}^\e(q)|<\infty,\ \ \ \ \ q \in\,\mathbb{R}^2.
\end{equation}
\end{Lemma}
\begin{proof}
We have
\[\frac{1}{\la^3(q)}-\frac{\la(q)}{(\la^2(q)+\e^2)^2}=\e^2\le[\frac{2\la^2(q)+\e^2}{\la^3(q)(\la^2(q)+\e^2)^2}\r],\]
and
\[\frac{1}{\la^3(q)}-\frac{\la(q)\le(\la^2(q)-\e^2\r)}{(\la^2(q)+\e^2)^3}=\e^2\le[\frac{4\la^4(q)+3\e^2\la^2(q)+\e^4}{\la^3(q)\le(\la^2(q)+\e^2\r)^3}\r]\]
and
\[\frac{1}{\la^2(q)}-\frac{\la^2(q)-\e^2}{(\la^2(q)+\e^2)^2}=\e^2\le[\frac{3\la^2(q)+\e^2}{\la^2(q)\le(\la^2(q)+\e^2\r)^2}\r].\]
Therefore, in view of the definition of $R^\e(q)$  in \eqref{sjy24}, \eqref{sjy25}, \eqref{sjy27} and \eqref{sjy28}, this concludes the proof.

\end{proof}

According to \eqref{sjy30}, \eqref{sjy40}, \eqref{sjy38} and \eqref{sjy39}, equation \eqref{sjy34} can be rewritten as
\begin{equation}
\label{sjy41}
\begin{array}{l}
\ds{dq_\e(q)=
\frac 1\e\frac{\beta_0(q_\e(t))}{\la^2(q_\e(t))}{\nabla}^\perp\la(q_\e(t))\,dt+B(q_\e(t))\,dt+\Sigma(q_\e(t))\,dw(t),}\\
\vs
\ds{+\e\,\le[B_\e(q_\e(t))\,dt+\Sigma_\e(q_\e(t))\,dw(t)\r],\ \ \ \ \ \ q_\e(0)=q,}
\end{array}\end{equation}
where
\[B(q)=\frac{1}{\la(q)}Ab(q)-M(q)\nabla \la(q),\ \ \ \Sigma(q)=\frac 1{\la(q)}A\si(q),\]
and
\[B_\e(q)=H^\e(q)b(q)+\frac 1\e\,{R}^\e(q)\nabla \la(q),\ \ \ \ \Sigma_\e(q)=H^\e(q)\si(q).\]This means that, as $\e\downarrow 0$, some of the coefficients are of order $O(\e^{-1})$, part of order $O(1)$ and part of order $O(\e)$.

With the notations introduced in the previous section, in what follows, we want to investigate the limiting behavior of the $\Gamma$-valued process $\Pi(q_\e(\cdot))=(\la(q_\e(\cdot)),k(q_\e(\cdot)))$, as $\e\downarrow 0$. 

Applying It\^o's formula to $\la(q_\e(t))$, we get
\[d\la(q_\e(t))=\mathcal{G}\la(q_\e(t))\,dt+\mathcal{A}\la(q_\e(t))\,dw(t)+\e\,\mathcal{G}_\e\la(q_\e(t))\,dt+\e\,\mathcal{A}_\e\la(q_\e(t))\,dw(t),\]
where for every $f \in\,C^2(\mathbb{R}^2)$ and $q \in\,\mathbb{R}^2$
\[\mathcal{G}f(q)=\frac 12\mbox{Tr}\le[\Sigma \Sigma^\star(q)D^2f(q)\r]+\langle Df(q),B(q)\rangle,\]
\[\mathcal{A}f(q)=\Sigma(q)^\star Df(q),\]
\[\mathcal{G}_\e f(q)=\frac 12 \mbox{Tr}\le[\le(\e\,\Sigma_\e \Sigma_\e^\star(q)+\Sigma\Sigma_\e^\star(q)+\Sigma_\e\Sigma^\star(q)\r)D^2f(q)\r]+\langle Df(q),B_\e(q)\rangle,\]
and
\[\mathcal{A}_\e f(q)=\Sigma_\e^\star(q) Df(q).\]

\medskip

We recall that the graph $\Gamma$ is made of $n$ edges $I_1,\ldots,I_n$ and $m$ vertices $O_1,\ldots,O_m$. For every $j=1,\ldots,n$ and for every $f$ that is twice differentiable in the interior of the edge $I_j$, we denote 
\begin{equation}
\label{sjy46}
\mathcal{L}_j f(z)=\frac 12 \a_j(z)f^{\prime \prime}(x)+\gamma_j(z)f^\prime(z),\end{equation}
where
\[\a_j(z)=\oint_{C_j(z)} |\mathcal{A}\la(x)|^2\,d\mu_{z,j}(x)=\oint_{C_j(z)} |\Sigma^\star(x)\nabla \la(x)|^2\,d\mu_{z,j}(x),\]
\[\gamma_j(z)=\oint_{C_j(z)} \mathcal{G}\la(x)\,d\mu_{z,j}(x),\]
and
$d\mu_{z, j}$ is the probability measure introduced in \eqref{brr}.

\begin{Definition}
\label{def1}
For each interior vertex $O_k$ and any edge $I_j$ adjacent to $O_k$ (notation $I_j\sim O_k$), let $\rho_{kj}$ be the positive constant defined by
\[\rho_{kj}=\oint_{C_{kj}}\frac{\la^2(x)}{\beta_0(x)|\nabla \la(x)|}|\Sigma^\star(x)\nabla \la(x)|^2\,dl(x).\]

We denote by $D(L)\subset C(\Gamma)$ the set consisting all continuously differentiable functions $f$ defined on the graph $\Gamma$ such that $\mathcal{L}_jf$ is well defined in the interior of the edge $I_j$ and for every $I_j\sim O_k$ there exists finite
\[\lim_{x\to O_k}\mathcal{L}_j f(x)\]
and the limit is independent of the edge $I_j$. Moreover, for each interior vertex $O_k$ 
\[\sum_{j\,:\,I_j\sim O_k} \pm \rho_{kj}f_j^\prime(\la(O_k))= 0,\]
where $f_j^\prime$ denotes the derivative of $f$ with respect to the local coordinate $z$, along the edge $I_j$ and the sign $\pm$ are taken if on the edge $I_j$ it holds $z>\la(O_k)$ or $z<\la(O_k)$.

Next, for every $f \in\,D(L) $, we  define
\[L f(x)=\begin{cases}
\mathcal{L}_j f(x),  &  \text{if}\ x\ \text{is an interior point of}\ I_j,\\
\lim_{x\to O_k}\mathcal{L}_j f(x),   &  \text{if}\ x\ \text{is the vertex}\ O_k\ \text{and}\ I_j\sim O_k.
\end{cases}\]

\end{Definition}

As proven in \cite[Theorem 8.2.1]{fw}, in case $\Sigma(q)=I$ the operator $L$ defined on the domain $D(L)$, as described in Definition \ref{def1}, is the generator of a strong Markov process $Y_t$ on $\Gamma$ with continuous trajectories. Here the same result holds, because of the non-degeneracy condition \eqref{sjy51} satisfied by the diffusion coefficient $\Sigma(q)$.

In fact, as shown in the next theorem, the Markov process $Y$ is the weak limit in $C([0,T];\Gamma)$ of the slow motion $\Pi(q_\e(\cdot))$ on $\Gamma$.

\begin{Theorem}
Under Hypotheses \ref{H1}, \ref{H2} and \ref{H3}, for every fixed $T>0$ the $\Gamma$-valued process $\Pi(q_\e(\cdot))$ converges weakly in $C([0,T];\Gamma)$ to the Markov process $Y$ generated by the operator $(L,D(L))$, introduced in Definition \ref{def1}.
\end{Theorem}

\begin{proof}
If in equation \eqref{sjy41} we have $B(q)=B_\e(q)=\Sigma_\e(q)=0$ and $\Sigma(q)=I$, the result above is what is proven in \cite[Theorem 8.2.2]{fw}. In the present situation we are dealing with the more general situation in which we have a coefficient  $B(q)$ of order $O(1)$ and coefficients $B_\e(q)$  of order $O(\e)$. Moreover we allow a non-constant diffusion coefficient $\Sigma(q)+\Sigma_\e(q)$, where $\Sigma(q)$ is of order $O(1)$ and $\Sigma_\e(q)$ is of order $O(\e)$. As shown in \cite{Y}, under these more general assumptions, an averaging principle of the same type of the one described in \cite[Theorem 8.2.2]{fw} is still valid. This of course has required to introduce a suitable generalization of the operator $(L, D(L))$, that takes into account the coefficients $B$ and $\Sigma$, and to extend the limiting result in presence of the vanishing terms $B_\e$ and $\Sigma_\e$.

\end{proof}

\end{document}